\documentclass{amsart}
\usepackage{euscript,verbatim, todonotes}           
\usepackage{amsmath,amsthm}     
\usepackage{amssymb}            

\DeclareMathOperator{\Map}{Map}
\newlength{\labwidth}
\newcommand{\labarrow}[1]{
\settowidth{\labwidth}{$\scriptstyle \;\; #1 \;\;$}
\stackrel{#1}{\smash{\hbox to \labwidth{\rightarrowfill}}
\vphantom{\longrightarrow}}
}

\renewcommand{\S}{{\mathbb S}}

\renewcommand{\O}{{\cal O}}

\renewcommand{\O}{{\mathcal O}}

\usepackage {epsf}
\DeclareMathOperator*{\colim}{colim}

\newcommand {\F}{{\mathbb F}}
\newcommand {\Z}{{\mathbb Z}}
\newcommand {\Q}{{\mathbb Q}}

\newcommand {\W}{{\mathbb W}}

\newcommand {\C}{{\mathbb C}}
\newcommand {\G}{{\mathbb G}}

\newcommand{\lra}{\longrightarrow}              

\newcommand{\cat}{\mathcal}    

\newfont{\german}       {eufm10 at 12pt}
\DeclareMathOperator{\Hom}{Hom}

\numberwithin{equation}{section}
\newtheorem{thm}{Theorem}[section]
\newcounter{numerierer}
\newcounter{leer}

\newtheorem{defn}[thm]{Definition}

\newtheorem{prop}[thm]{Proposition}

\newtheorem{cor}[thm]{Corollary}
\newtheorem{lemma}[thm]{Lemma}

\theoremstyle{definition}  
\newenvironment{definition}{\begin{defn}\rm}{\end{defn}}

\newtheorem{set theory}[thm]{Set Theoretic Prelude}

\newtheorem{convention}[thm]{Convention}

\newtheorem{remark}[thm]{Remark}
\usepackage[frame,matrix,curve, arrow]{xy}

\xymatrixcolsep{1.9pc}                          
\xymatrixrowsep{1.9pc}
\newdir{ >}{{}*!/-5pt/\dir{>}}                  
\newdir{ |}{{}*!/-5pt/\dir{|}}                  

\subjclass{}
\begin{document}

\title{Cannibalistic Classes of string bundles}
\author{Gerd Laures and Martin Olbermann}

\address{ Fakult\"at f\"ur Mathematik,  Ruhr-Universit\"at Bochum, NA1/66, D-44780 Bochum, Germany}


\subjclass[2000]{Primary 55N34; Secondary 57R20, 55N22}

\date{\today}

\begin{abstract}
We introduce cannibalistic classes for string bundles with values in $TMF$ with level structures. This allows us to compute the Morava $E$-homology of any map from the bordism spectrum $MString$ to $TMF$ with level structures.  

\end{abstract}

\maketitle

\section{Introduction and statement of results}
Suppose $E$ is a cohomology theory and $V$ is a bundle over a  space $X$ which is equipped with a Thom class  $\tau$ with respect to $E$. Then for each stable operation $\psi^g$ in $E$-theory one obtains a cannibalistic class $\theta^g(V)$ by the formula 
\begin{eqnarray} \psi^g (\tau)&=&\theta^g (V) \tau.\end{eqnarray}
For example, for singular cohomology with coefficients in $\F_2$ the cannibalistic class $\theta^k$ associated to the Steenrod square $Sq^k$ is given by the $k$th Stiefel-Whitney class. Another example arises in real and complex   $K$-theory for Adams operations (see \cite{MR0153019}). In the complex case, the cannibalistic classes of the canonical  line bundle $L$ over $\C P^\infty$ can be computed via the formula
\begin{eqnarray}\label{Euler} \theta^g(L)&=& \frac{\psi^g (x)}{x}\in \Z[\!]x]\!]\end{eqnarray}
where $x$ is the Euler class.
The cannibalistic classes  play an important role in the investigation of bordism invariants and spherical fibrations.\par
In this paper, we are interested in cannibalistic classes for string bundles with values in the spectrum of topological modular forms with various level structures or in Morava $E$-theory. Here, the  `string'-structure for a bundle refers to a spin structure together with a lift of the classifying map to the next higher connective cover $BO\langle 8\rangle$ of $BSpin$. The infinite loop space $BO\langle8\rangle$ is also called $BString$ and its associated Thom spectrum is called $MString$.  For a more detailed introduction to string bordism and the spectrum of topological modular forms $TMF$, the reader is referred to the survey articles \cite{MR1989190} and \cite {MR2648680}.\par
We will show how the cannibalistic classes for string bundles  are related to the  characteristic classes defined in \cite{MR3448393}. This enables us  to compute the $E$-homology of maps form  the string bordism spectrum $MString$ to the spectrum of topological modular forms with level structures.  \par
We will  now describe the results in more detail. In the $K(2)$-local category at the prime 2 the spectrum $TMF(3)$ of topological modular forms with respect to the level structure $\Gamma(3)$ coincides with the Morava $E$-theory spectrum $E_2$ and there is a continuous action of the Morava stabilizer group $\G_2$ on $E_2$. Moreover, string manifolds have a natural Thom class with values in $TMF(3)$ and in all of its fixed point spectra for finite subgroups of $\G_2$. This is  known as the Witten orientation \cite{AHR10}. Hence, there are stable cannibalistic classes $\theta^g$ for string bundles and for all elements of $\G_2$. \par
In \cite{MR3448393}, all $TMF_i(3)$-characteristic classes for string bundles have been computed for $i=0,1$. It turned out that they are generated by Pontryagin classes and one more class $r$ which, in the complex setting, measures the difference between the complex orientation from the elliptic curve and the Witten orientation. 

We compute the action of $\G_2$ on the class $r_U$ and its restriction $r_K$  to $K(\Z,3)$. It turns out that the action is given by the formula
$$\psi^g(r_K)=r_K^{\det(g)}$$
where $\det$ is the determinant.
This formula is related to results of Peterson \cite{Pet11}, Westerland \cite[Theorem 3.21]{W13}, Hopkins and Lurie \cite{HL13}. However, our proof is independent and self-contained.

This formula implies the calculation of the cannibalistic class $\theta^g$ for arbitrary string bundles. They are determined by the formula

\begin{eqnarray} \label{formula 1}
 (\theta^g )^2 & = &c^* (q_0^g \theta_\C^g) r^{det(g)-1} .
\end{eqnarray}

Here, $c$ is the complexification map, $\theta_\C^g$ is the cannibalistic class associated to the complex orientation satisfying Formula \ref{Euler} above and, finally,  $q_0^g$ is a specific $SU$-characteristic class. It describes the difference between the cubical structures of the underlying elliptic curves and its perturbation by $g$. It will be computed in more detail in section 3. 
 We will prove a splitting principle for string bundles which reduces the calculation of the string cannibalistic classes to  (\ref{formula 1}).\par
The formula allows us to compute the Morava $E$-homology of any map from the Thom spectrum $MString$ to $TMF_i(3)$. In \cite{MR3448393} and \cite{MR3471093} it was shown that the cohomology of $MString$ is topologically  freely generated  by $TMF_i(3)$-Pontryagin classes and the class $r$ described above.  The precise formulas in terms of the  dual generators  are given in section \ref{homology}.
\par
These results will prove useful in the investigation of indecomposable summands  of  string bordism. We think that in the $K(2)$-local category there is an additive splitting of $MString$ similar to the Anderson-Brown-Peterson splitting in which $TMF$ and $TMF$ with level structures appear as direct summands.
The computations in this work are all very explicit and can help to find such a splitting. However, the only class which prevents us to set up the splitting formula is the class $q_0$ described above. It seems to be hard to write down in a closed form in high dimensions. 
\subsubsection*{Acknowledgements.} The authors would like to thank Craig Westerland  and Thomas Nikolaus for helpful discussions. They are also grateful to the referee for useful hints to clarify and streamline some issues.

\section{$TMF$ with level structure and characteristic classes}
In this section we recall the definition of the characteristic classes in the spectrum of topological modular forms with level structures. We refer the reader to the articles \cite{MR2508904}, \cite{MR3471093}  and \cite{MR3448393} for an account on the spectra 
$$T_i=TMF_i(3)$$
for  $i\in \{ \emptyset, 0,1\}$. \par
Consider the supersingular elliptic curve in Weierstrass from
$$C: \quad y^2 +y=x^3$$  
over $\F_4$. There is a Lubin-Tate spectrum $E_2$ (also called Morava $E$-theory) associated to the formal group $\hat{C}$ with coefficients in the ring of power series over the Witt vectors of $\F_4$
$$E_2{}_*=\W(\F_4)[\! [u_1]\!][u^\pm].$$
Let $\G_2$  be the automorphism group of $\hat{C} \to Spec\  \F_4$ and let $G$ be the one of $C\to Spec\  \F_4$ (i.e. automorphisms are commutative squares). 
There is an isomorphism $\G_2 \cong \S_2 \rtimes Gal(\F_4:\F_2)$, where the subgroup $\S_2$  corresponds to commutative diagrams
\begin{eqnarray*}\label{}
&\xymatrix{
\hat{C} \ar[d]\ar[r]^-g
&
\hat{C}\ar[d]
\\
Spec\ \F_4
\ar[r]^-=
&
Spec\ \F_4,
}
\end{eqnarray*} 
and the Galois group acts via the Frobenius, which gives a commutative diagram
\begin{eqnarray*}\label{}
&\xymatrix{
\hat{C} \ar[d]\ar[r]^-\sigma
&
\hat{C}\ar[d]
\\
Spec\  \F_4
\ar[r]^\sigma
&
Spec\  \F_4.
}
\end{eqnarray*} 
\bigskip

Abstractly,  the group $\S_2$ can be described as the group of units in the maximal order $$End(\hat{C})=\O_2=\Z_2\{\omega,i,j,k\}$$ of the 2-adic quaternion algebra $$End(\hat{C})\left[\frac12\right]=D_2=\Q_2\{1,i,j,k\}.$$
Here $\omega=\frac12(-1-i-j-k)$ such that $\omega^2=-\omega -1,\omega^3=1$.
One computes that $\omega i \omega^{-1}=j, \omega j \omega^{-1}=k, \omega k \omega^{-1}=i$.
Note that we do not make a notational difference between the endomorphisms of $\hat{C}$ and the endomorphisms of the Honda formal group law; these rings are isomorphic. A more detailed account on the structure of the Morava stabilizer group can be found in \cite{MR3450774}.

The group $G$ is a maximal finite subgroup of $\G_2$ and has order 48. 
The automorphisms  $-1,\omega,i,j,k$ arise from automorphisms of $C$ over the identity of $\F_4=\F_2[\alpha]$:
\begin{itemize}
\item $-1$ corresponds to mapping each point on the curve to its negative, i.e. $x\mapsto x,y\mapsto y+1$,
\item $\omega$ corresponds to $x\mapsto \alpha x,y\mapsto y$, and to the automorphism $g(x)=\alpha x$ of the formal group,
\item $i$ corresponds to $x\mapsto x+1, y\mapsto x+y+\alpha$.
\end{itemize}
The Frobenius map $\sigma$ is the generator of $Gal(\F_4:\F_2)$, it acts via $x\mapsto x^2, y\mapsto y^2$ on $C$
as an automorphism over the Frobenius map of $\F_4$.
 
Note that $G\cong G_{24}\rtimes Gal(\F_4:\F_2)$, where $G_{24}$ is generated by $-1,\omega,i,j,k$,
and  $G_{24}\cong Q_8\rtimes C_3$, where $C_3$ is generated by $\omega$ and $Q_8=\{\pm 1,\pm i,\pm j, \pm k\}$.

The group of $\F_4$-points of $C$ is isomorphic to $(\Z/3)^2$. Via the induced action of $G$ we obtain an isomorphism
$$G\cong GL_2(\Z/3).$$ 
Choosing as generators the $\F_4$-points $(0,0)$ and $(1,\alpha)$, this isomorphism sends
$$-1\mapsto \begin{pmatrix}-1 &0\\0&-1 \end{pmatrix},\quad \omega\mapsto \begin{pmatrix}1 &1\\0&1 \end{pmatrix},\quad i\mapsto \begin{pmatrix}0 &-1\\1&0 \end{pmatrix}, 
\quad \sigma\mapsto  \begin{pmatrix}1 &0\\0&-1 \end{pmatrix}.$$
In particular $SL_2(\Z/3)$ corresponds to the group $G_{24}$, the subgroup $G_0$ of $GL_2(\Z/3)$ which fixes the subgroup $\Z/3\times 0$ corresponds to the group generated by $-1,\omega, \sigma$ and 
the group $G_1$ which fixes the point $(1,0)$ of $(\Z/3)^2$ corresponds to the group generated by $\omega$ and $\sigma$.

The following result is a version of the Goerss-Hopkins-Miller theorem (see \cite[Chapter 12]{MR3223024}).

\begin{thm} 
Let $K(2)$ be the second Morava $K$-theory at the prime 2.  
There is a continuous action $\psi$  of $G$ on $E_2$ and  canonical isomorphisms
\begin{eqnarray*}
L_{K(2)}TMF&\cong& E_2^{hG},\\
L_{K(2)}T_0&\cong& E_2^{hG_0},\\
L_{K(2)}T_1&\cong& E_2^{hG_1},\\
L_{K(2)}T&\cong& E_2.\\
\end{eqnarray*}
\end{thm}
\begin{convention}
In the sequel we will work in the $K(2)$-local category unless otherwise stated. We will omit the localization functor from the notation. In particular, a product $X\wedge Y$ denotes the $K(2)$-localization of the standard product.
\end{convention}
The following result can be found in \cite{MR1333942} and with more details in \cite{MR1763961} and \cite{MR2076002}.
\begin{prop} 
An element $f\in \pi_k (E_2 \wedge T_i)$ gives rise to the continuous function 
from $\G_2/G_i$ to $\pi_k E_2$ which sends $g$ to  the composite of $f$ with 
$$ \xymatrix{E_2\wedge T_i \ar[rr]^{ 1\wedge \psi^g  }&&E_2\wedge T_i\ar[r] & E_2 \wedge E_2\ar[r]& E_2.}$$  
This map is an isomorphism
\begin{eqnarray*}
\phi: \pi_* (E_2 \wedge T_i)& \stackrel{\cong}{\lra} &  \Map_{cts}(\G_2/G_i , \pi_* E_2)
\end{eqnarray*}
\end{prop}
In this notation, the operations satisfy
\begin{eqnarray*}
(1\wedge \psi^\nu)f (g)& =& f(g \nu )\\
( \psi^\nu \wedge 1 )f (g)& =& \psi^\nu  f(\nu^{-1}g  )
\end{eqnarray*}

Locally at the prime 2, for the theory $T_1$ it has been shown in \cite{MR3471093} that there are unique classes $p_i\in T_1^{4i}BSpin$ with the following property: the formal series  $p_t=1+p_1t+p_2t^2\ldots$ is given by $$ \prod_{i=1}^m (1+t\rho^*(x_i \overline{x}_i))$$
when restricted to the classifying space of each maximal torus of $Spin(2m)$. Here, $\rho$ is the map to the standard maximal torus of $SO(2m)$ and the $x_i$ (and  $\overline{x}_i$) are the first $T_1$-Chern classes of the canonical  line bundles $L_i$ (resp.\ $\overline{L}_i$) induced by the circle factors of the torus. 
The classes $p_i$ freely generate  the $T_1$ cohomology of $BSpin$, that is, 
$$ T_1^*BSpin\cong T_1^*[\! [p_1,p_2,\ldots ]\!].$$
It was shown in \cite{MR3471093} that the Kitchloo-Laures-Wilson sequence implies  an isomorphism of algebras
$$ {T}_1^*BString\cong {T}_1^*[\! [\tilde{r},p_1,p_2,\ldots ]\!] $$
where $p_1,p_2,\ldots$ are the Pontryagin classes coming from $BSpin$ and $\tilde{r}$ we be explained in more details below. 
\par
In \cite{MR3448393} the result has been used to study the $T_0$ cohomology of $BString$ using equivariant methods. The theory $T_1$ admits the structure of a Real spectrum in the sense of Atiyah, Kriz, Hu et al.  (compare \cite{MR0206940}\cite{MR1808224}).
This
means that there is a $\Z/2$-equivariant spectrum (``the Real theory'') whose non-equivariant restriction (``the complex theory'') is $T_1$
and whose fixed point spectrum (``the real theory'') is $T_0$. 
They were used in \cite{MR3448393} to show that there are classes  $\pi _i\in T_0^{-32i}BSpin$ which lift the products $a_3^{6i}p_i$ of the $T_1$-Pontryagin classes $p_i$ and the (invertible) $\Gamma_1(3)$-modulare form $a_3$. Moreover, we have an isomorphism 
$$ T_0^{*}BSpin\cong T_0^*[\![ \pi _1,\pi _2,\ldots]\!]$$
In order to obtain a class $r$ which is already defined in the real theory $T_0$ one has to provide a more geometric construction. Here, the theory of cubical structures on elliptic curves comes in and furnishes a construction of the  Witten orientation in \cite{ MR1869850}. It turns out that a convenient choice of a generator is possible in the connective cover $BU\langle 6 \rangle$ of $BU$ where we define
$$r_U:=\frac{re^*\sigma}{x}\in T_1^*BU\langle 6\rangle$$ as the difference class which compares the Witten orientation $$MU\langle 6 \rangle\stackrel{re}{\to}MString\stackrel{\sigma}{\to} T_1$$ with the complex orientation
$x\in T_1^*MU\langle 6\rangle$ described in \cite[p513f]{MR3448393}. 
For each stable complex vector bundle over a finite CW-complex $X$ equipped with a lift of its structure map $X\to BU$ to $\xi:X\to BU\langle 6\rangle$, we obtain a characteristic class 
$r_U(\xi)\in T_1^*X$ by pulling back $r_U$.

The space $BU\langle 6\rangle$ has an $H$-space structure induced from $BU$, and the inclusion of the homotopy fiber of $BU\langle 6 \rangle \to BSU$ is an $H$-space map $j:K(\Z,3)\to BU\langle 6\rangle$.
For two bundles $\xi:X\to BU\langle 6\rangle$, $\eta:X\to BU\langle 6\rangle$, we obtain a direct sum $\xi\oplus \eta:X\to BU\langle 6\rangle$ such that
$$r_U(\xi\oplus\eta)=r_U(\xi)r_U(\eta),$$
since both Thom classes are multiplicative. When restricted to $K(\Z,3)$ the class $\tilde{r}_K:=\tilde{r}_U(j)\in \tilde{T}_1^*K(\Z,3)$ generates the cohomology topologically
\begin{eqnarray}\label{CohKZ3}
T_1^*K(\Z ,3)& \cong & T_1^*[\![\tilde{r}_K]\!].
\end{eqnarray}

\begin{remark}
The class $r_K$ can be described in a second way, which was suggested to the authors by Thomas Nikolaus:
in the diagram
\begin{eqnarray}\label{}
&\xymatrix{
K(\Z,3)
\ar[r]^j\ar[d]^{=}
&
BU\langle 6 \rangle
\ar[r]\ar[d]^{re}
&
BSU\ar[d]^{re}
\\
K(\Z,3)
\ar[r]^i
&
BString 
\ar[r]
&
BSpin,
}&
\end{eqnarray}
pull the universal bundle over $BSpin$ back and take Thom spaces everywhere to obtain  
\begin{eqnarray}\label{}
&\xymatrix{
K(\Z,3)_+
\ar[r]^{Th(j)}\ar[d]^{=}
&
MU\langle 6 \rangle
\ar[r]\ar[d]^{re}
&
MSU\ar[d]^{re}
\\
K(\Z,3)_+
\ar[r]^{Th(i)}
&
MString 
\ar[r]
&
MSpin,
}&
\end{eqnarray}
Here the composition of the complex orientation with $Th(j)$ factors through $MSU$, so $K(\Z,3)_+\to MU\langle 6 \rangle \stackrel{x}{\to}  T_1$ is the unit element in $T_1^*K(\Z,3)$.
Thus the composition $$K(\Z,3)_+\to MU\langle 6 \rangle\stackrel{re}{\to}MString\stackrel{\sigma}{\to} tmf\to T_0\to T_1$$ is equal to the pull-back $r_K=j^*\frac{re^*\sigma }{x}$: 
$K(\Z,3)_+\to BU\langle 6\rangle\to T_1$.
In particular we see a natural lift of $r_K$ to a class in $T_0^*K(\Z,3)$, and in fact even a lift to $tmf^*K(\Z,3)$.
Note that
 \begin{eqnarray}\label{Thomclass}
 Th(i)^*\sigma&=&Th(j)^*(re^*\sigma)=r_K.
\end{eqnarray}
\end{remark}

The complex conjugation on $BU$ induces an involution on $BU\langle 6\rangle$. We denote the composition of $\xi$ with this complex conjugation on $BU\langle 6\rangle$ by $\overline{\xi}$. We write $\psi^{-1}$ for the standard involution on $T_1$ which changes the Euler class $x$ of a line bunlde to $[-1](x)$. In \cite{MR3448393} it is shown that $r_U$ can be turned into a  $\Z/2$-equivariant map.
It follows that $$r_U(\overline{\xi})=\psi^{-1}(r_U(\xi)).$$

\bigskip

Since $c: BString\to BU\langle 6\rangle$ maps to the fixed points of complex conjugation on $BU\langle 6\rangle$, the pull-back $r=r_U(c)$ admits a lift to the fixed point spectrum $T_0$ of $\psi^{-1}$.
By pulling back $r$ along the classifying map, there is a natural stable class $$r(\xi) \in T_0^0 X$$  for every string bundle $\xi$ over $X$. 
\begin{thm}\cite{MR3448393}\label{2.4}
The class $r$ has 
the following properties:
\begin{enumerate}
\item
$r$ is multiplicative: $r(\xi\oplus \eta)=r(\xi)\otimes r(\eta)$.
\item
There is an isomorphism
$$ {T_0}^*[\![\tilde{r}, \pi _1,\pi _2,\ldots]\!] \lra {T_0}^*BString$$
where  $\tilde{r}=r-1$ is the reduced version of the class $r$ corresponding to the universal bundle over $BString$.
\item
In terms of the Chern character of its
elliptic character $\lambda$ (cf.\cite{MR1022688}) at the cusp $\infty$, it is given by the formula
$$ ch(\lambda (r_U(\xi))) = \prod_i\frac{\Phi(\tau,x_i-\omega)}{\Phi(\tau,-\omega)}.
$$
Here, the $x_i$'s are the formal Chern roots of $\xi\otimes \C$, $\omega =2\pi i/3$ and $\Phi$ is the theta function
\begin{eqnarray*}
\Phi(\tau,x)&=&(e^{x/2}-e^{-x/2})\prod_{n=1}^\infty \frac{(1-q^ne^x)(1-q^ne^{-x})}{(1-q^n)^2}
\\&=&x\, \exp(-\sum_{k=1}^\infty \frac{2}{(2k)!}G_{2k}(\tau)x^{2k}).
\end{eqnarray*}\end{enumerate}
\end{thm}
\begin{remark}
Since the diagram
\begin{eqnarray}\label{}
&\xymatrix{
BU\langle 6 \rangle
\ar[r]^-{(1,conj)}\ar[d]^{re}
&
BU\langle 6 \rangle\times BU\langle 6 \rangle
\ar[d]^\mu
\\
BString 
\ar[r]^c
&
BU\langle 6 \rangle
}&
\end{eqnarray}
commutes,
we obtain the formulas
\begin{eqnarray}
r_U(c\circ re)&=&r(re)=r_U\cdot \psi^{-1}r_U, \label{core}\\
r(i)&=&r(re\circ j)=r_U(j)\cdot \psi^{-1}r_K=r_K^2,
\end{eqnarray}
where the last equality holds since $r_K$ comes from a $T_0$-cohomology class and is thus $\Z/2$-invariant.
\end{remark}
\begin{thm}
\begin{enumerate}
\item
Let  $ b_i \in \pi_{2i}( E_2 \wedge BS^1)$ be the dual to the power ${c_1^i}$ of the first Chern class of the canonical line bundle. Denote its image  under the map induced by the inclusion of the maximal torus
$$\xymatrix{\iota: BS^1\ar[r]& BSpin(3)\ar[r]&BSpin}$$
by the same name. Then 
$$
\pi_*E_2\wedge BSpin_+ \cong \pi_*E_2 [b_{8}, b_{16},b_{32},\ldots].
$$
\item
Denote a lift of $b_i $ to $\pi_{2i} (E_2\wedge BString)$  by $a_i$. Then we have
$$ \pi_*E_2\wedge BString_+ \cong \pi_*(E_2\wedge K(\Z,3)) [a_{8}, a_{16},a_{32},\ldots].$$
\end{enumerate}
\end{thm}
\begin{proof}
This follows from \cite{MR1909866}\cite{MR2093483} and \cite{MR584466}.
\end{proof}
We want to be more specific in choice of the classes $a_i$. Let 
$$ R^*: E_2^*BString \cong E_2^*BSpin [\![ \tilde{r}]\! ] \lra E_2^*BSpin$$ be the ring homomorphism given by the constant coefficient in the power series expansion. Then dually, we get a map
$${R}_* : (E_2)_*BSpin  \lra (E_2)_*BString.$$

\begin{lemma}\label{small}
Let $a_i$ be the image of $b_i$ under this map. Then we state the equality
$$\left< a_i, c\right> = \left< b_i,  R^*(c)\right>$$
for later purposes.
\end{lemma}

\section{The action of the Morava stabilizer group}
In this section we analyze the action of the Morava stabilizer group on  $E$-theory. 
The first three sections summarize some well known facts about the action of the Morava stabilizer group (see for example \cite[Section 4]{MR3127044} or \cite[Section 6]{MR3581316}). We denote the formal group law obtained from the coordinate $z=-\frac{x}{y}$ on $\hat{C}$ by $F$, then elements of
$\S_2=Aut(F)$ are certain power series $g(t)\in\F_4[[t]]$.

\subsection{The action on the formal group for Morava E-theory}

One can lift the curve $C$ to $$C_U:y^2+3u_1xy+(u_1^3-1)y=x^3$$ over $(E_2)_0=\W(\F_4)[[u_1]]$,
whose formal group law $F_U$ is a universal deformation of $F$. 

An automorphism of $F$ (i.e. a power series $g$ over $\F_4$) can be lifted to a power series $\tilde{g}$
with coefficients in $\W(\F_4)$. Then $$(F_U^{\tilde{g}})(x,y) = \tilde{g}^{-1}(F_U(\tilde{g}(x),\tilde{g}(y))).$$
This is isomorphic, but in general not $*$-isomorphic to $F_U$. It is a deformation of $F$,
so it is classified by a ring homomorphism 
$$h^g: (E_2)_0=\W(\F_4)[[u_1]]\to (E_2)_0=\W(\F_4)[[u_1]],$$
which is characterized by the property that $(h^g)_*F_U$ is $*$-isomorphic to $F_U^{\tilde{g}}$.

The composition $g_U\in (E_2)_*[[z]]$ of the isomorphism of formal group laws $$\tilde{g}:F_U \to F_U^{\tilde{g}}$$ with
the unique $*$-isomorphism $$F_U^{\tilde{g}} \to (h^g)_*F_U$$ is the unique lift $g_U$ of $g$ such that $F_U^{g_U}$ 
can be obtained by pushing forward $F_U$ by a ring isomorphism, namely $h^g$. 

It follows that for each $g\in \S_2$ we have a commutative square
\begin{eqnarray*}\label{}
&\xymatrix{
\hat{C}_U \ar[d]\ar[r]^-{g_U}
&
\hat{C}_U\ar[d]
\\
Spec\ (E_2)_0
\ar[r]^{h^g}
&
Spec\ (E_2)_0
}
\end{eqnarray*} 

We denote $$g_U(z)=t_0(g)z+t_1(g)z^2+t_2(g)z^3+\dots$$

\subsection{The action on Morava E-theory}
We are now prepared to describe the action on the coefficients $(E_2)_*$. We are forced to extend $h^g:(E_2)_0\to (E_2)_0$ by
$$
h^g (u) = t_0(g)u
$$
to an automorphism $h^g$ of $(E_2)_*$.

In order to describe the action 
of $\S_2$ on $(E_2)_*X$ we briefly leave the $K(2)$-local category. Recall that $E_2$ is a Landweber exact theory and hence
$$(E_2)_*X = (E_2)_*\otimes_{MU_*} MU_*X.$$ 
Here, the $MU_*$-module structure of $(E_2)_*$ comes along as follows: the ring $MU_*$ carries the universal graded fomal group law and there is  a graded formal group law $\overline{F}_U$ over $(E_2)_*$ defined by the equation
$$\overline{F}_U(x,y)=u^{-1}F_U(ux,uy).$$
The isomorphism $g_U:F_U \to (h^g)_*F_U$ induces an isomorphism $\overline{g}_U:\overline{F}_U \to (h^g)_*\overline{F}_U$ of 
graded formal group laws via
$$\overline{g}_U(z)=(h^g(u))^{-1}g_U(uz).$$ 
This
is classified by a ring homomorphism $\phi: MU_*MU\to (E_2)_*$ sending $t_i\in MU_*MU$ to the $(i+1)$-st coefficient
$t_i(g)t_0(g)^{-1}u^{i}\in (E_2)_{2-2i}$ of $\overline{g}_U$. 

The map $\phi\eta_L:MU_*\to E_2$ classifies $F_U$ and $\phi\eta_R:MU_*\to E_2$ classifies $(h^g)_*F_U$.
Using the $MU_*MU$-coaction $$\psi:MU_*X\to MU_*MU\otimes_{MU_*} MU_*X,$$ the action of $g$ on $(E_2)_*X$ is now given as
\begin{align*}
(E_2)_*\ ^{\phi\eta_L}\otimes_{MU_*} MU_*X \stackrel{h^g\otimes \psi}\longrightarrow (E_2)_*\ ^{\phi\eta_R}\otimes_{MU_*}^{\phi\eta_R} MU_*MU\otimes_{MU_*} MU_*X \\
\stackrel{1\otimes \phi\otimes 1}\longrightarrow
(E_2)_*\ ^{\phi\eta_R}\otimes_{MU_*}^{\phi\eta_R} (E_2)_*\ ^{\phi\eta_L}\otimes_{MU_*} MU_*X \stackrel{m\cdot 1}\longrightarrow (E_2)_*\ ^{\phi\eta_L}\otimes_{MU_*} MU_*X.
\end{align*}
We have defined, for each $g\in \S_2$, a natural transformation $(E_2)_*(-)\to (E_2)_*(-)$ of homology theories, and therefore
a map of spectra $E_2\stackrel{\psi^g}\to E_2$ up to homotopy as there are no phantom maps (see for example \cite[Theorem 4c]{MR1235295}). 

Vice versa the action of the stabilizer group on $E_2$-homology is given by sending a homology class
$x\in (E_2)_*X$, i.e. $S\stackrel{x}\to E_2\wedge X$ to $\psi^g(x)$ defined as $S\stackrel{x}\to E_2\wedge X\stackrel{\psi^g\wedge 1}\to E_2\wedge X$,
and the action on a cohomology class $X\to E_2$ is given by composition with $\psi^g:E_2 \to E_2$.
This has the property that with the Kronecker pairing $E_2^*X\times (E_2)_*X \to (E_2)_*$ we have
$$\psi^g\langle b,x\rangle =\langle  \psi^g b,  \psi^g x \rangle.$$

\subsection{The action on the $E$-Euler class}

The action of $\S_2=Aut(\hat{C})$ on the complex coordinate $x\in (E_2)^2\C P^\infty$
is given by $$g\cdot x=\overline{g}_U(x)=\sum_{i\ge 0} t_i(g)t_0(g)^{-1}u^{i}x^{i+1}.$$
For $z=ux\in (E_2)^0\C P^\infty$ we have $g\cdot z =g_U(z)$.

Beaudry shows in \cite{MR3450774} that $$g\cdot u_1=t_0(g)u_1+\frac{2t_1(g)}{3t_0(g)}.$$
The element $\omega$ 
corresponds to the power series $g(t)=\alpha t$, which one can lift to $\tilde{g}(t)=\omega t$.
The automorphism $h^\omega$ defined by 
\begin{align*}
\omega\cdot u_1&=\omega u_1
\end{align*}
satisfies $(h^{\omega})_*\tilde{F}=\tilde{F}^{\tilde{g}}$, i.e. the push-forward is in this case not only $*$-isomorphic, but equal to $\tilde{F}^{\tilde{g}}$.
In particular we have 
\begin{align*}
\omega\cdot u &=\omega u,\\ 
\omega\cdot x &= x. 
\end{align*}

Strickland  uses the compatibility of the $GL_2(\Z/3)$-action on $T^*(\C P^\infty)$ with the $\G_2$-action on $E_2^*\C P^\infty$
to prove that \cite[Section 2.4]{MR3450774}
\begin{align*}
i\cdot u     &=  u \frac{-1-2\omega}{u_1-1},\\ 
i\cdot u_1 & = \frac{u_1+2}{u_1-1}.
\end{align*}

Moreover we have $$i\cdot z=\frac{lz+rlw(z)}{1+sz+l^3(sr-t)w(z)},$$
where
$$l=\frac{-1-2\omega}{u_1-1},\quad r=3\frac{1-u_1^3}{(u_1-1)^3}, \quad s=3\frac{\omega^2u_1-1}{u_1-1}, \quad 
t=3\frac{u_1^3-1}{(u_1-1)^4}\left(1-\omega+(1-\omega^2)u_1\right)$$
and $w(z)\in(E_2)_*[[z]]$ is the power series with leading term $z^3$ given by solving 
$$w+3u_1zw+(u_1^3-1)w^2=z^3$$
for $w$.

The action of the Frobenius is trivial on $u,u_1$, but is the conjugation on $\W(\F_4)$.

\subsection{The action on the $E$-cohomology of $K(\Z,3)$}

After the description of the operation on a complex coordinate in $E_2^*K(\Z,2)$, we can address the more complicated
operation on the generator of $E_2^*K(\Z,3)$ (see Equation \ref{CohKZ3}). The following result is related to \cite{Pet11}  and \cite[Theorem 3.21]{W13} where the analogous formula for the homology is proven. The result probably can also be recovered from \cite{HL13}.

\begin{thm} \label{action K}The $\S_2$-action on
$E_2^*K(\Z,3)= E_2^*[\! [\tilde{r}_K]\! ] $
is given by 
\begin{eqnarray}\label{FKZ3}
\psi^g (r_K) & = &r_K^{\det(g)}
\end{eqnarray}
where $\det$ is the determinant. 

\end{thm}
The proof will be provided in several steps. We will write $R$ for $\pi_0(E_2)$, $\pi$ for its maximal ideal and $k$ for the residue field $R/\pi$.
\begin{lemma}\label{Omegaact}
The Formula \ref{FKZ3} holds for $g=\omega$. In other words, $\psi^\omega$ acts trivially on $r_K$.
\end{lemma}
 \begin{proof}
We claim that the action of $\omega$ is already trivial on $r_U$. The class $r_U$ is the  the difference class of the complex and the string Thom classes. Each of the Thom class are already defined  in $E_2^{hG_1}$ and are hence fixed under the action of $\omega$.

\end{proof}
\begin{lemma}
There is a homomorphism $\alpha: \S_2 \lra \Z_2^\times$ with 
$$\psi^g (r_K) = r_K^{\alpha(g)}.$$
\end{lemma}
\begin{proof}
The class $r_U$ is multiplicative for sums of vector bundles. Hence,
the $H$-space addition $\mu: BU\langle 6\rangle^2 \to BU\langle 6\rangle$ sends 
$$r_U\mapsto r_U\otimes r_U\in E_2^*BU\langle 6\rangle\otimes_{E_2^*}E_2^*BU\langle 6\rangle$$
in cohomology. 
Since $j:K(\Z,3)\to BU\langle 6\rangle$ is an $H$-space map which maps $r_U\mapsto r_K$ in cohomology, 
the $H$-space addition $\mu: K(\Z,3)^2\to K(\Z,3)$ sends 
$$r_K\mapsto r_K\otimes r_K\in E_2^*K(\Z,3)\otimes_{E_2^*}E_2^*K(\Z,3)$$ 
in cohomology.
Let $\psi^g(r_K)=q(\tilde{r}_K)$ for $q\in R[\! [t]\! ]$ a power series with leading term 1.
Then 
\begin{align*}
q(\tilde{r}_K)\otimes q(\tilde{r}_K) =\psi^g(r_K)\otimes\psi^g(r_K)=\psi^g(\mu^*(r_K))=\mu^*\psi^g(r_K)\\
=\mu^*(q(\tilde{r}_K))=q(\mu^*(\tilde{r}_K))=q(\tilde{r}_K\otimes 1 + 1\otimes \tilde{r}_K + \tilde{r}_K\otimes\tilde{r}_K)
\end{align*}
Thus $q(t)$ is a power series with leading term 1 satisfying 
\begin{equation} \label{powere}
q(t_1)q(t_2)=q(t_1+t_2+t_1t_2).
\end{equation}
For $\alpha\in R$ there cannot exist more than one power series with leading terms $1+\alpha t $ satisfying the equation. The proof is an easy 
induction and a comparison of  the degree $n$ coefficients on both sides of the equation.
This is true even if we invert 2, when all binomial series 
$$B_\alpha(t)=\sum{\alpha\choose n}t^n; \qquad \mbox{ } \binom{\alpha}{n}= \frac{\alpha (\alpha -1)\cdots (\alpha-n+1)}{n!}$$ for all $\alpha\in R$ satisfy this equation. 
It follows that 
$$\psi^g(r_K)=q(\tilde{r}_K)=B_\alpha(\tilde{r}_K)= (1+\tilde{r}_K)^\alpha= r_K^{\alpha}$$ for some exponent $\alpha=\alpha(g)$. Here, the third equation is by definition. \par
We still have to show that $\alpha$ already takes values in  $\Z_2^\times$. 
In order to see this, observe that all binomial coefficients $\binom{\alpha}{n}$ lie in the non rationalized coefficient ring $R$.
In particular  for $n=2$ it follows that  
$$\frac{\alpha (\alpha -1)}{2!}\in R.$$
Since $ \F_4[\! [u_1]\!]$ is an integral domain we conclude that 
$ \alpha$ coincides with an integer mod $2$. Similarly, for $n=2^2$ one obtains that $\alpha$ coincides with an integer modulo $2^2$ and so on. 
\par Since $r_K$ is multiplicative  $\alpha$ is  a homomorphism.
\end{proof}

It remains to show that this homomorphism is the determinant.
\begin{lemma}
Suppose $\alpha$ and $ \beta $ are $p$-adic integers and suppose
$$B_\alpha (t) = B_\beta  (t)\quad \mbox{mod } p$$
then $\alpha = \beta$.
\end{lemma}

\begin{proof}
Suppose $\alpha = \beta +p^s\gamma$ for some integer $s$ and some $p$-adic number $\gamma$. Then it sufices to show that $\gamma$ is divisible by $p$. Calculate mod $p$
$$
(1+t)^\alpha =(1+t)^\beta ((1+t)^{p^s})^\gamma =(1+t)^\beta (1+t^{p^s})^\gamma$$
and hence
$$ 1=(1+t^{p^s})^\gamma = 1+ \gamma \,  t^{p^s}+ \ldots. $$
\end{proof}
\begin{lemma}\label{MKZ3}
The Formula \ref{FKZ3}
 holds if and only if it holds modulo the maximal ideal $\pi$ of $R$.
\end{lemma}
\begin{proof} This is a consequence of the previous two lemmas.
\end{proof}
In order to simplify the statement even further, consider the product 
$$\mu_n: K(\Z/n,1)\times  K(\Z/n,1) \lra K(\Z/n, 2)\stackrel{\beta}\to K(\Z,3)$$
which is derived from the product structure of the Eilenberg-MacLane spectrum and the Bockstein map.  
The map $\mu_n$  is additive in each of the two arguments and the ring 
$E_2^0(K(\Z ,3))$ supports the multiplicative formal group. This fact is not difficult to show. A proof is given in  \cite[Proposition 3.2]{W13} for Morava $K$-Theory.  
Since $ E_2^0(K(\Z ,3) $ is a deformation of $K ^0(K(\Z ,3)$ it also supports the multiplicative group.

The additivity can be reformulated in the language of formal group schemes. Let ${\cat A}dic$ be the category of augmented $R$-algebras  with nilpotent augmentation ideal. For an augmented $R$-algebra  $A$  with augmentation ideal $m$ write $spf(A)$ for the functor from ${\cat A}dic$ to ${\cat S}ets$ given by 
$$spf(A)(B)=\colim_i {\cat A}dic(A/m^i,B).$$
 Write $G[n]$ for the  the formal group scheme $spf(E_2^0(K(\Z/n ,1)))$. Then 
$\mu_n$ induces a pairing\footnote{In fact, this map induces the universal $e_n$-pairing in the sense of Ando and Strickland \cite[Proposition 2.3]{MR1791270} but we will not make use of this fact. } 
$$ f: G[n]\times G[n] \lra spf(E_2^0(K(\Z,3)))$$
with the properties 
\begin{eqnarray*}
f(x_1+x_2,y)&=&f(x_1,y)f(x_2,y),\\
f(x,y_1+y_2)&=&f(x,y_1)f(x,y_2),\\
f(x,x)&=&1.
\end{eqnarray*}

The left action of the Morava stabilizer group on the cohomology induces a right action on formal schemes via precomposition. Writing $g$ instead of $\psi^g$  we have the identity
\begin{eqnarray}
\label{ringop}
 f(xg,yg)&=&f(x,y)g.
 \end{eqnarray}
Furthermore, Lemma \ref{Omegaact} gives the equality
\begin{eqnarray}\label{trivial}
 f(x {\omega},y {\omega})&=&f(x,y).
 \end{eqnarray}

Let $A\in {\cat A}dic$ be an algebra over the residue field $k=R/\pi$. 
Then the set of $A$-points $G[n](A)$ has the form 
$$spf(E_2^0(K(\Z/n ,1)))(A)\cong spf(K^0(K(\Z/n ,1)))(A)\cong 
spf(k[\![x]\! ]/[n](x))(A).
$$ 
This observation allows us to extend the action of $\S_2$ on $G[n](A)$ to all power series which are endomorphism of $\hat{C}$. Hence, for all $x\in G[n](A)$  it holds by definition
$$ x ({g+h})=xg +xh.$$
At this point we need the well-known presentation of the endomorphism ring
$$End(\hat{C})=\Z_2[\omega]\langle S\rangle /(S^2=2,Sa=\overline{a}S)$$
where $S$ is the Frobenius and the conjugation satisfies $\bar{\omega}=\omega^2$. In this description, the determinant  takes the form
$$ \det (a+bS)=a\bar{a}-2b\bar{b}.$$

\begin{lemma}\label{3equ}
Formula  \ref{FKZ3} is equivalent to each of the following statements:
\begin{enumerate}
\item $f(x,xS)g=f(x,x S)^{\det(g)}$ for all $x\in G[n](A)$ 
\item $f(x,y)g=f(x,y)^{\det(g)}$ for all $x,y\in G[n](A)$ 
\end{enumerate}
\end{lemma}
\begin{proof}
Ravenel and Wilson \cite{MR584466} show that the maps
$$ \colim_rK(Z/2^r, s)\simeq K(Z/2^\infty , s)\lra  K(\Z, s+1)$$
are $2$-local equivalences. It follows that the maps 
$$ E_2^* K(\Z,s+1) \lra \lim_r E_2^*K(\Z/2^r, s)$$
are isomorphisms and thus $\lim E_2^*K(\Z/2^r, 2) $ and 
$\lim  E_2^*K(\Z/2^r, 1)$ are power series rings in one variable.
Hence, in order to show that the limit of  the composites
$$ E_2^*K(\Z/2^r, 2) \stackrel{\mu^*}{\lra} E_2^*K(\Z/2^r, 1)\otimes E_2^*K(\Z/2^r, 1)
\stackrel{(id,S)}{\lra} E_2^*K(\Z/2^r, 1).$$
is an injection it suffices to show that the generator $r_K$ maps non trivially.  This can be checked in Morava $K$-theory where $r_K$ pairs non trivially with the homology class
$$\beta_2\circ (\beta_2 S)=\beta_2\circ \beta_1.$$
Here, the circle product is induced by the map $\mu$. This calculation has been carried out in \cite[p521ff]{MR3448393}. We conclude that Formula \ref{FKZ3} can be shown either in $E_2^*K(\Z/2^r, 1)$ or in $ E_2^*K(\Z/2^r, 1)\otimes E_2^*K(\Z/2^r, 1)$ for each $r$.
Moreover, we have already seen in Lemma \ref{MKZ3} that we can reduce Formula  \ref{FKZ3} modulo the ideal $\pi$.
\par
Now suppose  that  statement $(i)$ holds. Set $A=K^0(K(\Z/2^r,1))/m^j$ for some $j$ and let $x$ be the ($m^j$, $\pi$)-reduction map. Then we can apply the algebra map $f(x,xS)g$ to $r_K$ and obtain the image of $\psi^g(r_K)$ in $A$ under the map described above. Similarly, we obtain the image of $r_K^{\det(g)}$ when applying $f(x,xS)^{\det(g)}$ to $r_K$. Since these coincide for all $j$ we have shown Formula \ref{FKZ3} in $K^0K(\Z/2^r, 1)$. 
\par
We now turn to the converse. Ravenel and Wilson show that the map from  $K_0(K(\Z/2^r,2)$ to $K_0(K(\Z,3))$ is injective. Since the dual map is surjective the image of the class $r_K$ generates.  Hence, it suffices to check the stated equality $(i)$ on this single class for which it follows from Formula \ref{FKZ3}.
\par
Statement $(ii)$ is treated analogously.
 \end{proof}

In view of Lemma \ref{3equ}
 the proof of Theorem \ref{action K} follows from the subsequent lemma.
\begin{lemma}
For all $x\in G[n](A)$ and $g\in \S_2$ we have
\begin{eqnarray}\label{actionf}
f(x,xS)g&=&f(x,xS)^{\det(g)}.
\end{eqnarray}
\end{lemma}

\begin{proof}
For $g=a+bS$ compute
\begin{eqnarray*}
f(x,xS)g&=&f(xg,xSg)\\ 
&=&f(x({a+bS}),x({\bar{a}S+2\bar{b}}))\\
&=& f(x{a}+x{bS},x{\bar{a}S}+x{2\bar{b}})\\
&=&  f(x{a},x{\bar{a}S}) f(x{a},x{2\bar{b}})  f(x{bS},x{\bar{a}S})  f(x{bS},x{2\bar{b}})
\end{eqnarray*}
We claim that the middle two factors cancel each other and that the rest gives
$
f(x,xS)^{\det(g)}
$
as desired.
For all $g$ with $a,b\in \Z$ this follows easily from the properties of  the pairing $f$. 
In particular for $1+S$, we have proved with the previous lemma the equality
\begin{eqnarray*}\label{actionfxy}
f(x,y)(1+S)&=&f(x,y)^{-1}.
\end{eqnarray*}

 Using also (\ref{ringop}) and (\ref{trivial}) we compute:
\begin{eqnarray*}
f(x{\omega},x) 
&=&f(x{\omega}(1+S),x(1+S))^{-1} = f(x(1+S),x\omega (1+S))\\
&=&f(x+xS,x{\omega}+x{\omega S})\\
&=&f(x,x{\omega})f(x,x{\omega S})f(xS,x{\omega})f(xS,x{\omega S})\\
&=& f(x,x{\omega})f(x,xS\overline{\omega})f(xS,x{\omega})f(xS,x{\omega S})\\
&=& f(x,x{\omega})f(x\omega,xS)f(xS,x{\omega})f(xS,x{\omega S})\\
&=& f(x,x{\omega})f(xS,x{\omega S})
\end{eqnarray*}
and thus
$$ f(x{\omega},x)^2=f(x{S},x{\omega S}).$$
Write $a=a_1+a_2\omega$ and $b=b_1+b_2 \omega$. Then the first middle factor is
\begin{eqnarray*}
f(xa,x2\bar{b})&=& f(xa_1+xa_2\omega,2xb_1+2x\bar{\omega}b_2)\\
&=& f(x,x\bar{\omega})^{a_12b_2}f(x\omega,x)^{a_22b_1}f(x\omega,x\bar{\omega})^{a_22b_2}\\&=&f(x\omega,x)^{2(a_1b_2+a_2b_1-a_2b_2)}
\end{eqnarray*}
and the second one is
\begin{eqnarray*}
f(xS(b_1+b_2\bar{\omega}),xS(a_1+a_2\omega))&=&f(xS,xS\omega)^{b_1a_2}f(xS\bar{\omega},xS)^{a_1b_2}f(xS\bar{\omega},xS\omega)^{b_2a_2}\\
&=&f(xS,xS\bar{\omega})^{-b_1a_2-a_1b_2+b_2a_2}\\
&=&f(xS,x\omega S)^{-(a_1b_2+a_2b_1-b_2a_2)}.
\end{eqnarray*}
Hence, they cancel. It remains to compute the first and the last term.
\begin{eqnarray*}
f(xa,x\bar{a}S)&=&f(xa_1+xa_2\omega, xa_1S+xa_2S\omega)\\
&=&f(x,xS)^{a_1^2}f(x,xS\omega)^{a_1a_2}f(x\omega, xS)^{a_2a_1}f(x\omega,xS\omega)^{a_2^2}\\
&=&f(x,xS)^{a_1^2}f(x \bar{\omega}+x\omega, xS)^{a_1a_2}f(x,xS)^{a_2^2}\\
&=& f(x,xS)^{a_1^2-a_1a_2+a_2^2}\\
&=& f(x,xS)^{a\bar{a}}
\end{eqnarray*}
The identity
\begin{eqnarray*}
f(x{bS},x{2\bar{b}})&=& f(x,xS)^{-2b\bar{b}}
\end{eqnarray*}
follows similarly. 
\end{proof}

\subsection{The action on $r_U$}
In this subsection we describe the action on the difference class $r_U$ in the cohomology of $BU\langle 6\rangle$. We then also know the action on the string characteristic class $r$ because
 it is the image of $r_U$ under the map induced by the complexification map.
 \begin{thm} \label{action rU}The $\S_2$-action on
$$r_U\in E_2^*BU\langle 6 \rangle $$
is given by $$\psi^g (r_U) = q_0^g r_U^{det(g)}$$
for a unique $q_0^g\in E_2^*BSU$.
\end{thm}
\begin{proof}
When we set $\psi^g(r_U)=q(\tilde{r}_U)$ for $$q=\sum q_j t^j\in E_2^*(BSU)[[t]]$$ a power series,
we do not have the equality (\ref{powere}) any more, but instead
\begin{equation} \label{mpowere}
q(t_1)q(t_2)=(\mu^*q)(t_1+t_2+t_1t_2)
\end{equation}
where $\mu^*q$ is obtained by pulling back the coefficients of $q$ via $\mu$. Also $q$ does not need to have leading term 1 any more.
Comparing coefficients of $t_1^j$ and $t_2^j$, we obtain $q_j\otimes q_0 = q_0\otimes q_j$ for all $j$.
Since $E_2^*BSU$ is a direct product of copies of $E_2^*$ and $q_0$ is a primitive class, 
this equation can only hold if $q_j$ is a multiple of $q_0$ by an element in $E_2^*$. So $\psi^g(r_U)$ is $q_0$ times 
a power series in $E_2^*[[\tilde{r}_U]]$. Under the map to $E_2^*K(\Z,3)$, the class $q_0$ maps to $1$ and the
power series in $E_2^*[[\tilde{r}_U]]$ is not changed, so the statement of the theorem follows from \ref{action K}.
\end{proof}
\begin{cor}\label{action r}
$\psi^g(r)=c^*(q_0^g)r^{det(g)}$
\end{cor}
\begin{proof}
This follows immediately from the Theorem and from the equation $$c^*r_U=r.$$
\end{proof}
We devote the rest of this section to the investigation of the class $q_0$ for $SU$-bundles.
\begin{prop}
Let
$f:  BU(1)^3\to BU\langle 6\rangle $ be the map which classifies the product $(1-L_1)(1-L_2)(1-L_3)$.Then we have
\begin{eqnarray*}
f^*q_0^g&=& \frac{1-3((t_0(g)^2u_1+\frac{2}{3}t_1(g))^3-t_0(g)^3) u^3x_0x_1x_2+\dots}
{1-det(g)3(u_1^3-1)u^3x_0x_1x_2+\dots}
\end{eqnarray*}
\end{prop}
\begin{proof}
We have $$f^*r_U=1-3(u_1^3-1)u^3x_0x_1x_2+\dots$$
and
\begin{eqnarray*}
f^*q_0^g(1-3(u_1^3-1)u^3x_0x_1x_2+\dots)^{det(g)}= f^*\psi^g(r_U)\\
=\psi^g(f^*r_U) = 
1-3(\psi^g(u_1)^3-1)\psi^g(u)^3\psi^g(x_0)\psi^g(x_1)\psi^g(x_2)+\dots .
\end{eqnarray*}
This yields
\begin{eqnarray*}
f^*q_0^g&=&\frac{1-3(\psi^g(u_1)^3-1)\psi^g(u)^3\psi^g(x_0)\psi^g(x_1)\psi^g(x_2)+\dots}
{1-det(g)3(u_1^3-1)u^3x_0x_1x_2+\dots}\\
&=& \frac{1-3((t_0(g)u_1+\frac{2t_1(g)}{3t_0(g)})^3-1) t_0(g)^3u^3x_0x_1x_2+\dots}
{1-det(g)3(u_1^3-1)u^3x_0x_1x_2+\dots}\\
&=& \frac{1-3((t_0(g)^2u_1+\frac{2}{3}t_1(g))^3-t_0(g)^3) u^3x_0x_1x_2+\dots}
{1-det(g)3(u_1^3-1)u^3x_0x_1x_2+\dots}.
\end{eqnarray*}
\end{proof}
The class $q_0$ is multiplicative. Hence, it can be described by a single power series which we work out next.

\begin{definition}
Let $K(\Z,3)\to P\to BS^1=K(\Z,2)$ be the fiber bundle with $k$-invariant a generator of $H^4(BS^1;\Z)$.
\end{definition}
Let $\kappa$ be the composition $BS^1\to BSU(2)\to BSU$, where the first map is induced by the inclusion of 
a maximal torus into $SU(2)$.
Then there are the following pullback squares:
$$
\xymatrix{P\ar[r]^-\iota\ar[d]^p&  BString \ar[d]\\  BS^1 \ar[r]^-\iota &BSpin, }\qquad
\xymatrix{P\ar[r]^-\kappa\ar[d]^p&  BU\langle 6\rangle \ar[d]\\  BS^1 \ar[r]^-\kappa &BSU. }
$$

\begin{lemma}
There is an isomorphism 
$$ E_2^*P \cong E_2^* [\! [x,\tilde{r}_P]\!] $$
where $x$ is the pullback of the class with same name in $E_2^*BS^1$ and $r_P=\kappa^*r_U$.  
As a virtual real vector bundle, 
the bundle $\iota$ over $P$ is isomorphic to the pullback of $L^2 -1_\C$ over $BS^1$ to $P$.
\end{lemma} 
\begin{proof}
For the fibration $K(\Z, 3)\lra P\lra BS^1$ the generalized Atiyah-Hizebruch spectral sequence 
$$H^*(BS^1, E_2^*K(\Z,3))\Longrightarrow E_2^*P$$
collapses since it is concentrated in even degrees (compare \cite[Proposition 2.0.1]{MR1648284}). The
class $\tilde{r}_P=\kappa^*\tilde{r}_U$ restricts to a generator of the fiber.
The last statement follows from the commutative diagram 
\begin{eqnarray}\label{Pdiagram}
\xymatrix{
P  \ar[rr]\ar[d]&&BString  \ar[d] \\
BS^1 \ar[d]_-{B(z\mapsto z^2)}\ar[r] & BSpin(3)\ar[d]\ar[r] & BSpin\ar[d] \\
BS^1\ar[r] & BSO(3)\ar[r] & BSO.
}
\end{eqnarray}
\end{proof}
\begin{prop}\label{formula q0}
Rationally, the elliptic Chern character of the multiplicative $SU$-class $q_0^g$ 
for the bundle $\xi=(1-L_1)(1-L_2)$ over $BS^1\times BS^1$ is given by
$$q_0^g (\xi)=\psi^g\left( \frac{\beta(\tau ,x+y)}{\beta(\tau ,x)\beta(\tau ,y))}\right) 
\left( \frac{\beta(\tau ,x)\beta(\tau ,y)}{\beta(\tau ,x+y)}\right)^{det(g)}.$$
Here, $\beta$ is the power series
$$\beta(\tau,x)=\frac{\Phi(\tau,x-\omega)}{\Phi(\tau,-\omega)}$$
and $\Phi$ is the Weierstrass $\Phi$-function (cf.\ref{2.4}(iii)).
In particular for the bundle $\kappa =  L+\bar{L}-2 =-(1-L)(1-\bar{L})$ we obtain

$$ q_0^g (\kappa)=\frac{ (\beta(\tau ,x)\beta(\tau ,-{x}))^{det(g)}  }{\psi^g(\beta(\tau ,x)\beta(\tau ,-{x}))} . $$
\end{prop}
\begin{proof}
Rationally, the maps
$$ BU\langle 6 \rangle \lra BSU \lra BU$$
induce ring maps in cohomology which are compatible with the action of $\G_2$ and which send the 1-structure to a 2-structure via 
$$\delta(g)(x,y)=\frac{g(x+_Fy)}{g(x)g(y)}$$
and similarly a 2-structure to a cubical structure over the coefficient ring (compare \cite{MR1869850}). Since the class $\beta(\tau, x)$ coincides with $r_U$ by Theorem \ref{2.4} 
the result follows from the identity 
$$ q_0^g=\frac{\psi^g(r_U)}{r_U^{det(g)}}$$
and the fact that $q_0^g$ is an $SU$-characteristic class.
\end{proof}

\section{Cannibalistic classes}
Before we consider cannibalistic classes with values in $TMF$ with level structures or in $E_2$  we will review cannibalistic classes in complex $K$-theory. Since we are only interested in stable operations it is convenient to work in $p$-adic $K$-theory. We write $\psi^q$ for the stable Adams operation if $q$ is a $p$-adic unit. Bott's cannibalistic classes
$\theta^q (V)\in K(X)$ for virtual vector bundles $V$ over $X$ are defined by the equation
$$ \psi^q (\tau) = \theta^q(V)\tau.$$
when  $\tau$ is the Thom class. These classes are multiplicative, that is, they take sums to products because the Thom classes and the operations behave that way. Hence, the splitting principle implies that it suffices to consider the canonical line bundle $L$ over $BS^1$. Let $s$ denote its zero section in the Thom space $Th(L)$. Then the composite 
$$ \xymatrix{ K(BS^1)\ar[r]^\tau & \tilde{K}^2(Th(L))\ar[r]^s &\tilde{K}^2(BS^1)}$$
is the multiplication with the Euler class $x=v^{-1}(1-L)$, ($v$ denotes the Bott class). Hence it is  injective.  Since the product $\theta^q(L)\, x$  coincides with $\psi^q(x)$ this gives
$$ \theta^q(L)=\frac{\psi^q(x)}{x}=\frac{[q](x)}{qx}=\frac{1-(1-vx)^q}{qvx}.$$ 
For real $p$-adic $K$-theory we assume that $V$ is a spin bundle. Again by the splitting principle, it suffices to compute the real cannibalistic classes of the spin bundle $L^2$ over $BS^1$. As in the complex case, the composite
$$ \xymatrix{ KO(BS^1)\ar[r]^\tau & \widetilde{KO}^2(Th(L^2))\ar[r]^s &\widetilde{KO}^2(BS^1)\ar[r]^c&\tilde{K}^2(BS^1)}$$
is the multiplication by the complexified Euler class $e$. The positive complexified spinor bundle of $L^2$ is $\bar{L}$ and the negative one is $L$. So, the Euler class is
$$ e=v^{-1}(\bar{L}-L)=x-\bar{x}.$$
and hence $\theta^q(L^2)\, e = \psi^q(x-\bar{x})$. The following result follows immediately form this computation. It can also be deduced from a calculation of  Adams in \cite{MR0198469}  with the Chern character.
\begin{prop}\cite[Lemma 3.14]{MR2004426}
$$c\, \theta^q(L^2)=\frac{(1-vx)^{-q}-(1-vx)^q}{q((1-vx)^{-1}-(1-vx))}.$$
\end{prop}
There is another way to compute the real cannibalistic class of $L^2$ which will be more instructive when it comes to the string case. For that, observe that the virtual spin bundle $L^2-1$ coincides with the realification of the $SU$-bundle
\begin{eqnarray}\label{su bdl}
 L^2-1+\bar{L}-L &=& (1-L)^2-(1-L)(1-\bar{L}).
 \end{eqnarray}
Moreover, the real Thom class of an $SU$- bundle coincides with the complex Thom class. Hence, when we write $\theta_U^q$ for the complex cannibalistic class we get
$$ c \, \theta^q(L^2-1)=\frac{\theta^q_U((1-L)^2)}{\theta_U^q((1-L)(1-\bar{L})}.$$
Now the result follows from an elementary calculation and the formula above for the complex cannibalistic classes. 
\par
We now turn to the string case.
\begin{definition} 
Let $\xi$ be a virtual string bundle of dimension 0 over some base space $X$. For $g \in \G_2$ define the cannibalistic class $\theta^g(\xi)\in E_2^*(X)$ by the equation
$$ \psi^g (\tau)=\theta^g (\xi) \tau $$
where $\tau$ is the Thom class obtained by pulling back $\sigma$ to the Thom space of $\xi$. 
We denote by ${\theta}^g=\theta^g(\sigma)$ the cannibalistic class for the universal stable bundle over $BString$.
In particular, we have $ \theta^g (\xi) = \xi^*{\theta}^g$.

Similarly, we have a complex Thom class for the universal complex vector bundle,
and so a complex Thom class $\tau_\C$ for every virtual complex vector bundle $\xi$ of dimension 0 over $X$.  This results in a complex cannibalistic class $$ \theta_\C^g (\xi)=\frac{\psi^g(\tau_\C)}{\tau_\C}\in E_2^*(X),$$ which is the pullback of the 
class $\theta_\C^g$ for the universal complex vector bundle.

\end{definition}
\begin{prop}\label{cannibalisticsum}
We have
\begin{enumerate}
\item
$\theta^g(0)=1$
\item 
$\theta^{g}(\xi +\eta)
=\theta^g (\xi )  \theta^g (\eta) 
$
\item 
$ \theta^{g \nu}(\xi) = \psi^\nu(\theta^g(\xi)) \theta^\nu (\xi)$.
\end{enumerate}

\end{prop}
The proof follows immediately from the definition. Analogous formulas hold for the complex cannibalistic classes. The complex classes can be calculated with the help of splitting principle and the following result which follows from the definition.
\begin{prop} For the complex line bundle $L$ over $BU(1)$ with Euler class $x\in E_2^0\Sigma^{-2}BU(1)$ we have 
$$ \theta_\C^{g}(L)=\frac{\psi^{g}(x)}{x}.$$

\end{prop}

The string class $\theta^g$ is much harder to calculate. 
We will first compute $(\theta^g)^2$.
Since the diagram
\begin{eqnarray}\label{}
&\xymatrix{
BString
\ar[r]^-{\Delta}\ar[d]^{c}
&
BString \times BString
\ar[d]^\mu
\\
BU\langle 6 \rangle 
\ar[r]^{re}
&
BString
}&
\end{eqnarray}
commutes, we obtain with part (ii) of \ref{cannibalisticsum} the equality
\begin{eqnarray}\label{square}
c^*\theta^{g}(re)&=&(\theta^g)^2.
\end{eqnarray}
Now write the string Thom class as a product $re^*\sigma=r_U\cdot x$, where
$x$ is the complex Thom class of $BU\langle 6\rangle \to BU$. The latter map can be used to pull back $\theta_\C^{g}$ to 
$BU\langle 6\rangle$.
\begin{cor} There is the identity
\label{Cor4.4}$$\theta^g (re) = q_0^g r_U^{det(g)-1}\theta_\C^g.$$
In particular, the cannibalistic class is the reduction of an $SU$-characteristic class if and only if the determinant of $g$ is 1.
\end{cor}
\begin{proof}
Calculate with Lemma \ref{action rU}
$$\theta^{g}(re)=\frac{\psi^{g}(r_Ux)}{r_Ux}=\frac{\psi^{g}(r_U)}{r_U}\theta_\C^{g}=q_0^g r_U^{\det (g)-1}\theta_\C^{g}.$$

\end{proof}
\begin{cor} \label{realmain}
$(\theta^g)^2  = c^*(q_0^g\theta_\C^g) r^{{det(g)-1}}.$
\end{cor}
\begin{proof}
The equality follows from equation (\ref{square}) and Corollary \ref{Cor4.4}.
\end{proof}
There is a more convenient way to describe the string cannibalistic classes which uses a splitting principle. 
\begin{prop}
Let $BT^\infty=\colim BT^n$ be a maximal torus of $BSpin$ and let 
$P^\infty = \colim P^n$ be the fibration over $BT^\infty$.
Then the induced map
$$ E_2^*BString  \lra  E_2^*P^\infty$$
is an injection.
\end{prop}
\begin{proof}
It follows from \cite[Corollary 3.8]{MR3471093} that the restriction from $E_2^*BSpin $ to $E_2^*BT^\infty$ is injective and hence so is the restriction to $P^\infty$ by the generalized Atiyah-Hirzebruch spectral sequence.
\end{proof}
Hence, it suffices to consider the bundle $L^2-1$ over $P$. 
We have seen in \ref{su bdl} that this bundle admits a reduction to the special unitary group. This gives
\begin{eqnarray}\label{Formula 4.5}
\theta^g(L^2-1)&=&\theta^g(re((1-L)^2-(1-\bar{L})(1-L)))\\
&=& (q_0^g r_U^{det(g)-1}\theta_\C^g)((1-L)^2-(1-\bar{L})(1-L))
\end{eqnarray}
The class $q_0$ can then be computed with  Formula \ref{formula q0}.
\section{The homology of string characteristic classes}\label{homology}
In this section we will use the calculation of cannibalistic classes to compute the  homology of string characteristic classes.
\begin{lemma}
For all $c\in T_i^* MString$ the  diagram 
$$ \xymatrix{\pi_* E_2 \wedge BString_+\ar[r]^-\rho &\Hom_{cts}(E_2^*BString, E_2{}_*)\\
\pi_* E_2 \wedge MString \ar[r]^-\rho \ar[u]_\varphi ^\cong \ar[d]^{(1\wedge c)_* }  &\Hom_{cts}(E_2^*MString, E_2{}_*)\ar[u]^\cong_\varphi  \ar[d]^{c_\sharp}  \\
\pi_* E_2 \wedge T_i \ar[r]^-\phi &\Map_{cts}(\G_2/G_i, E_2{}_*) }$$
commutes. Here, $\varphi $ is the Thom isomorphism, $\rho$ is the duality map and the lower right vertical arrow takes a homomorphism $b$ to the map
$$c_\sharp(b): g \mapsto b(MString \stackrel{c}{\lra} \Sigma ^*  T_i \lra \Sigma ^*  E_2 \stackrel{\psi^g}{\lra } \Sigma^*  E_2).$$
\end{lemma}
\begin{proof}
The diagram chase is readily verified.
\end{proof}
We denote by $c$ also the composition $MString \stackrel{c}{\lra} \Sigma ^*  T_i \lra \Sigma ^*  E_2$,
and we denote by $\tilde{c}$ its preimage under the Thom isomorphism $E_2^*BString\cong E_2^*MString$.
Set
$$\Theta_c\stackrel{def}{=} c_\sharp \varphi ^{-1} \rho=\phi (1\wedge c)_* \varphi^{-1}: \pi_*E_2\wedge BString _+\lra \Map_{cts}(\G_2/G_i,E_2{}_*).$$

\begin{prop}\label{canab}
For all $a \in \pi_* E_2\wedge BString _+$ we have
$$ \Theta_c(a)(g ) = \left< a, \theta^g\cdot  \psi^g (\tilde{c})\right>.$$

\end{prop}
\begin{proof} Set $b= \varphi^{-1} \rho (a)$. Then we have
\begin{eqnarray*}
 \Theta_c(a)(g ) &=& c_\sharp(b) (g) = b(\psi^g (\tau \tilde{c} )) \\
 &=&  b(\theta^g \cdot \psi^g (\tilde{c} )\tau )= \rho(a)(\theta^g \cdot \psi^g (\tilde{c} ))\\
&=&  \left< a, \theta^g \cdot \psi^g (\tilde{c})\right>
\end{eqnarray*}
\end{proof}

\begin{thm} Suppose $\tilde{c}$ is in the image of $E_2^*BSpin\to E_2^*BString$.
Then we have an equality of power series
$$\sum_i  \Theta_c(a_i)(g ) x^i = (q_0^g\theta_\C^g)^{\frac{1}{2}}((\bar{L}-L)^2) \psi^g(\tilde{c}(L^2-1))$$ 
Here, $x$ is the Euler class of the canonical bundle $L$ on $BS^1 $.
\end{thm}
\begin{proof}
We have with Proposition \ref{canab} and Lemma \ref{small}
\begin{eqnarray*}
\sum_i  \Theta_c(a_i)(g ) x^i & =& \sum_i \left< a_i, \theta^g\cdot  \psi^g (\tilde{c})\right>x^i
 \\
 &=& \sum_i  \left< b_i,  R^*(\theta^g)  R^*( \psi^g (\tilde{c})) \right> x^i\\
 &=&  \iota^* (R^*(\theta^g))\iota^*(  \psi^g (\tilde{c}) ).
 \end{eqnarray*}
Corollary \ref{realmain} implies
$$ \iota^*R^*(\theta^g)^2 = \iota^* c^*(q_0^g\theta^g_\C)^2 =q_0^g\theta^g_\C(\bar{L}-L)^2$$
and hence gives the formula.
\end{proof}
\bigskip

Next, we look at classes in the image of $i: K(\Z,3)\to BString$ in $E_2$-homology. Note that the universal string bundle is trivial when pulled back along $i$.  By equation (\ref{Thomclass}) the class $r_K$ is the image of the string Thom class under the map
$$ Ti: K(\Z,3)_+ \lra MString.$$   This gives
$$i^*\theta^g=i^*\frac{\psi^g \tau}{\tau}=\frac{\psi^g r_K}{r_K}={r}_K^{\det (g )-1}$$
Hence, we obtain for arbitrary $c$ 
$$
\Theta_c(i_*s)(g)= \langle i_*s, \theta^g \psi^g(\tilde{c})\rangle = 
\langle s, {r}_K^{\det (g )-1}\psi^{g}(i^*\tilde{c})\rangle 
$$
\begin{thm}
 Let $q_k \in (E_2)_{0}K(\Z,3)$ be the dual of $\tilde{r}_K^k$. Then it holds
 $$ \Theta_{\tilde{r}}(i_*(q_k))(g)=\binom{3\, \det(g)-1}{k}-\binom{\det(g)-1}{k}.
$$

\end{thm}
\begin{proof} The equality 
$$ \sum_k \Theta_{\tilde{r}}(i_*(q_k))(g)\tilde{r}_K^k=
{r}_K^{\det(g)-1}({r}_K^{2\det(g)}-1)
$$
follows from the fact that $i^*r=r_K^2$, Theorem \ref{action K} and the calculation above.

\end{proof}

\bibliographystyle{amsalpha}

\bibliography{toda}

\end{document}